\documentclass{amsart}

\usepackage{amsmath,amssymb,amsthm}

\usepackage{graphicx,tikz, caption}

\newtheorem{theorem}{Theorem}

\newtheorem{corollary}{Corollary}
\newtheorem*{proposition}{Proposition}

\newtheorem{lemma}{Lemma}

\theoremstyle{definition}

\theoremstyle{remark}

\begin{document}

\title[]{Some Remarks on the Erd\H{o}s \\Distinct Subset Sums Problem}

\author[]{Stefan Steinerberger}
\address{Department of Mathematics, University of Washington, Seattle, WA 98195, USA} \email{steinerb@uw.edu}

\thanks{S.S. is supported by the NSF (DMS-2123224) and the Alfred P. Sloan Foundation.}

\begin{abstract} Let $\left\{a_1, \dots, a_n\right\} \subset \mathbb{N}$ be a set of positive integers, $a_n$ denoting the largest element, so that for any two of the $2^n$ subsets the sum of all elements is distinct. Erd\H{o}s asked whether this implies $a_n \geq c \cdot 2^n$ for some universal $c>0$. We prove, slightly extending a result of Elkies, that for any $a_1, \dots, a_n \in \mathbb{R}_{>0}$
$$ \int_{\mathbb{R}} \left( \frac{\sin{ x}}{ x} \right)^2  \prod_{i=1}^{n} \cos{( a_i x)^2} dx \geq \frac{\pi}{2^{n}}$$
with equality if and only if all subset sums are $1$-separated. This leads to a new proof of the currently best lower bound $a_n \geq \sqrt{2/\pi n} \cdot 2^n$. The main new insight is that having distinct subset sums and $a_n$ small requires the random variable $X = \pm a_1 \pm a_2 \pm \dots \pm a_n$ to be close to Gaussian in a precise sense.
\end{abstract}

\maketitle

\section{Introduction}
 A problem of Erd\H{o}s \cite{erd} is as follows: if $\left\{a_1, \dots, a_n\right\} \subset \mathbb{N}$ is a set of positive integers, assumed to be ordered as $a_1 < a_2 < \dots < a_n$, such that for each of the $2^n$ subsets the sum of all elements is unique, does this force $a_n \geq c \cdot 2^n$ for some universal $c>0$? The problem is quite old. Erd\H{o}s \cite{erd3} refers to it as ``perhaps my first serious conjecture which goes back to 1931 or 32''.
  Since the sums over all subsets leads to $2^{n} -1$ distinct positive integers, one has
$ \sum_{i=1}^{n} a_i \geq 2^{n} -1$
(sharp for the powers of $2$) and $a_n \gtrsim 2^n/n$. Currently, the best known bound is
$$ a_n \geq (c-o(1)) \frac{2^n}{\sqrt{n}}$$
where different estimates for $c$ have been given over the years
\begin{align*}
c &\geq 1/4 \qquad &&\mbox{Erd\H{o}s and Moser \cite{erd}}\\
&\geq 2/3^{3/2} \qquad &&\mbox{Alon and Spencer \cite{alon}} \\
&\geq  1/\sqrt{\pi} &&\mbox{Elkies \cite{elk}}\\
&\geq 1/\sqrt{3} \qquad &&  \mbox{Bae \cite{bae} , Guy \cite{guy}}\\
&\geq  \sqrt{3/2\pi} &&\mbox{Aliev \cite{aliev}} \\
&\geq \sqrt{2/\pi} &&\mbox{Dubroff, Fox and Xu \cite{dubroff}}.
\end{align*}
The literature (see \cite{aliev}) mentions an unpublished manuscript of Elkies and Gleason also showing $c \geq \sqrt{2/\pi}$.
Dubroff, Fox and Xu give two different proofs: one appeals to the Berry-Esseen Theorem, the other uses an isoperimetric principle of
Harper \cite{harper}. 
In the other direction, we note that the powers of 2, with $a_n = 2^{n-1}$, are not extremal: already in 1968, Conway and Guy \cite{conway} (answering another question by Erd\H{os} \cite{erd2}) produced a candidate construction showing that $a_{n} \leq 2^{n-2}$ is possible (see Bohman \cite{boh}). The currently best construction is due to Bohman \cite{boh2} showing $a_n \leq 0.88008 \cdot 2^{n-2}$, see also \cite{behr, bor, lind, lunnon, malt}. It is an interesting question whether relaxing the condition somewhat can give rise to interesting examples. More concretely, are there sets $\left\{a_1, \dots, a_n\right\} \subset \mathbb{N}$ such that the subset sums attain $(1-o(1)) \cdot 2^n$ distinct values and $a_n = o(2^n)$?
\\
 
 The main purpose of our paper is to give a new proof of $c \geq \sqrt{2/\pi}$.  Many arguments, starting with Erd\H{o}s and Moser \cite{erd}, have considered the random walk $X = \pm a_1 \pm a_2 \pm \dots \pm a_n$, where all signs are chosen independently and uniformly at random. If all subset sums are distinct, then all $2^n$ possible outcomes of the random walk are equally likely and they are all at least distance 2 from each other.
A well-known argument (see \cite{bae,erd, guy, lind}) exploits this by using
$$ n \cdot a_n^2 \geq  \sum_{i=1}^{n} a_i^2 = \mathbb{E} (X^2) \geq \frac{2}{2^n}  \sum_{k=1}^{2^{n-1}}{(2k-1)^2} = \frac{4^n-1}{3}$$
  which shows $c \geq 1/\sqrt{3}$.
 This was further refined by Dubroff, Fox and Xu \cite{dubroff} who argued, using the Berry-Esseen theorem, that if $a_n^2$ is relatively small compared to $\sum_{i=1}^{n}a_i^2$ (the variance of the random walk), then the random walk is well-described by a Gaussian. Our argument will imply a somewhat converse result: unless the distribution of the random walk is close to a Gaussian (in a sense that will be made precise), the set cannot have distinct subset sums and $a_n$ small. This leads to an interesting reformulation of the Erd\H{o}s distinct subset sums problem as a problem in probability theory: whether it is possible for random walks with a large variance but relatively small largest stepsize to emulate a Gaussian distribution very well. 
\section{Results}
\subsection{Main Results.} We start with a basic analytic characterization of what it means for a set of $n$ positive real numbers to have the property that all subset sums are at least distance 1 from each other (if all numbers are integers, then this is the same as asking them to be distinct).
\begin{theorem}
Let $a_1, \dots, a_n > 0$ be positive, real numbers. Then
$$ \int_{\mathbb{R}} \left( \frac{\sin{2 \pi x}}{ 2 \pi x} \right)^2  \prod_{i=1}^{n} \cos{(2 \pi a_i x)^2} dx \geq \frac{1}{2^{n+1}}.$$
Equality occurs if and only if all subset sums are distance $\geq 1$ from each other.
\end{theorem}
This result is very similar to the analytic approach of Elkies \cite{elk} based on Laurent series. If all $a_i$ are integers, the product is
$2\pi$-periodic which simplifies the integral and recovers the characterization used by Elkies.
\begin{corollary}[Elkies \cite{elk}] Let $a_1, \dots, a_n > 0$ be positive integers. Then
$$ \int_0^1  \prod_{i=1}^{n} \cos{( 2\pi a_i x)^2} dx \geq \frac{1}{2^n}$$
with equality if and only if all subset sums are distinct.
\end{corollary}

All cosines in the product are aligned around $x=0$. A natural approach is thus to bound the contribution coming from a small interval around the origin of length $\sim 1/a_n^{}$. If $a_n$ is too small, that contribution is too large (see Lemma 1) and this was Elkies' original approach to prove $c \geq \sqrt{1/\pi}$ (see Lemma 1). The main novelty of our approach is to analyze the contribution coming from outside that interval. This leads to Corollary 2.

\begin{corollary}
We have
$$ a_n \geq \left(1-o(1)\right) \cdot  \sqrt{\frac{2}{\pi}}\frac{2^{n}}{\sqrt{n}}.$$
\end{corollary}

 While Corollary 2 itself does not tell us anything new, the proof establishes a connection to probability theory which will be discussed in \S 2.2 and \S 2.3.

\subsection{Proof of Corollary 2: Outline.} We use Theorem 1. The first ingredient is a lower bound on how much the integrand contributes to the integral close to the origin where all the cosines are aligned.

\begin{lemma}[see Elkies \cite{elk}]
Suppose that $\left\{a_1, \dots, a_n\right\}$ is a subset of the positive real numbers. Then
$$ \int_{|x| \leq \frac{1}{4a_n}} \left( \frac{\sin{2\pi x}}{2\pi  x} \right)^2  \prod_{i=1}^{n} \cos{( 2\pi a_i x)^2} dx \geq (1+o(1)) \cdot  \frac{1}{2} \frac{1}{ a_n} \frac{1}{\sqrt{ \pi n}}.$$
\end{lemma}

This Lemma in conjunction with Theorem 1 already shows $c \geq 1/\sqrt{\pi}$. The main new idea is to prove that contributions far away from the origin can also be analyzed and that they also contribute a substantial amount.

\begin{lemma}
Let $c>0$. Suppose that $\left\{a_1, \dots, a_n\right\} \subset \mathbb{R}_{>0}$ has 1-separated subset sums and, for some $\varepsilon>0$, we have $a_n^2 \leq c \cdot n^{-2/3 -\varepsilon} \sum_{i=1}^{n} a_i^2$. Then, as $n \rightarrow \infty$,
$$ \int_{|x| \geq \frac{1}{4a_n}} \left( \frac{\sin{2\pi x}}{2\pi  x} \right)^2  \prod_{i=1}^{n} \cos{( 2\pi a_i x)^2} dx \geq (1+o(1))\frac{\sqrt{2} - 1}{2 \sqrt{\pi}} \left( \sum_{i=1}^{n}a_i^2 \right)^{-1/2}.$$
\end{lemma}

We note that this lower bound can be bounded from below in terms of $a_n$ using the trivial bound $\sum_{i=1}^{n} a_i^2 \leq n \cdot a_n^2$. 
Combining this with Theorem 1 and Lemma 1, we see that if all subset sums are 1-separated, then
$$ \frac{1}{2^{n+1}} \geq (1+o(1)) \left( \frac{1}{2} \frac{1}{ a_n} \frac{1}{\sqrt{ \pi n}} +  \frac{\sqrt{2} - 1}{2 \sqrt{\pi}} \frac{1}{\sqrt{n} \cdot a_n} \right) = \frac{1+o(1)}{\sqrt{2\pi} \sqrt{n} \cdot a_n}$$
which shows $c \geq \sqrt{2/\pi}$. Lemma 2 appears to be very technical but contains an interesting idea which will tell us something new. Lemma 2 can be written in a completely different way (Theorem 2) and this alternative formulation is also how we are going to prove Lemma 2.

\subsection{Subset Sums and Gaussian Densities.} Let $A = \left\{a_1, \dots, a_n\right\} \subset \mathbb{R}_{>0}$ be a set of positive reals. As already indicated above, we consider the random variable
$ X = \sum_{i=1}^{n} \varepsilon_i a_i$
where $\varepsilon_i \in \left\{-1,1\right\}$ independently and with equal likelihood (also known as Rademacher random variables). This random variable is distributed according to some probability
measure $\mu$ on $\mathbb{R}$. Note that we can write
$$ X = - \sum_{i=1}^{n} a_i + 2 \sum_{i=1 \atop \varepsilon_i = 1}^{n} a_i.$$
If the minimal distance between the sum of two different subsets of $A$ is $1$, then the minimal distance between any two distinct values of $X$ is two. Moreover, by the subset sum condition, $X$ assumes $2^n$ distinct values which implies 
$$ \mu = \frac{1}{2^n} \sum_{i=1}^{2^n} \delta_{x_i} \qquad \mbox{where} \qquad \min_{i \neq j} |x_i - x_j| = 2.$$
Our main question of interest will now be whether $\mu$ is close to a Gaussian (and, if so, in what sense). Consider first a simple example: the set $\left\{1, 2, \dots, 2^{n-1}\right\}$. It is easy to see that all subset sums are distinct (the uniqueness of binary expansion) and, following the construction, we see that $\mu$ is supported on all $2^n$ odd numbers in $[-2^n, 2^n]$ roughly emulating a uniform distribution over that interval. A uniform distribution is not particularly close to a Gaussian overall. This will now be compared to a better construction: we take the first 22 terms induced by the Conway-Guy sequence \cite{conway} (where 22 was chosen so as to be `large' while still computationally feasible). We end up with a set $\left\{a_1,\dots, a_{22}\right\} \subset \mathbb{N}$ with distinct subset sums and $a_{22} = 1051905 \sim 0.51 \cdot 2^{21}$. The probability distribution of the associated random walk $\mu$ is shown in Figure 1. This is quite a bit closer to a Gaussian than uniform distribution would be. This is not a coincidence.

\begin{center}
\begin{figure}[h!]
\includegraphics[width=0.7\textwidth]{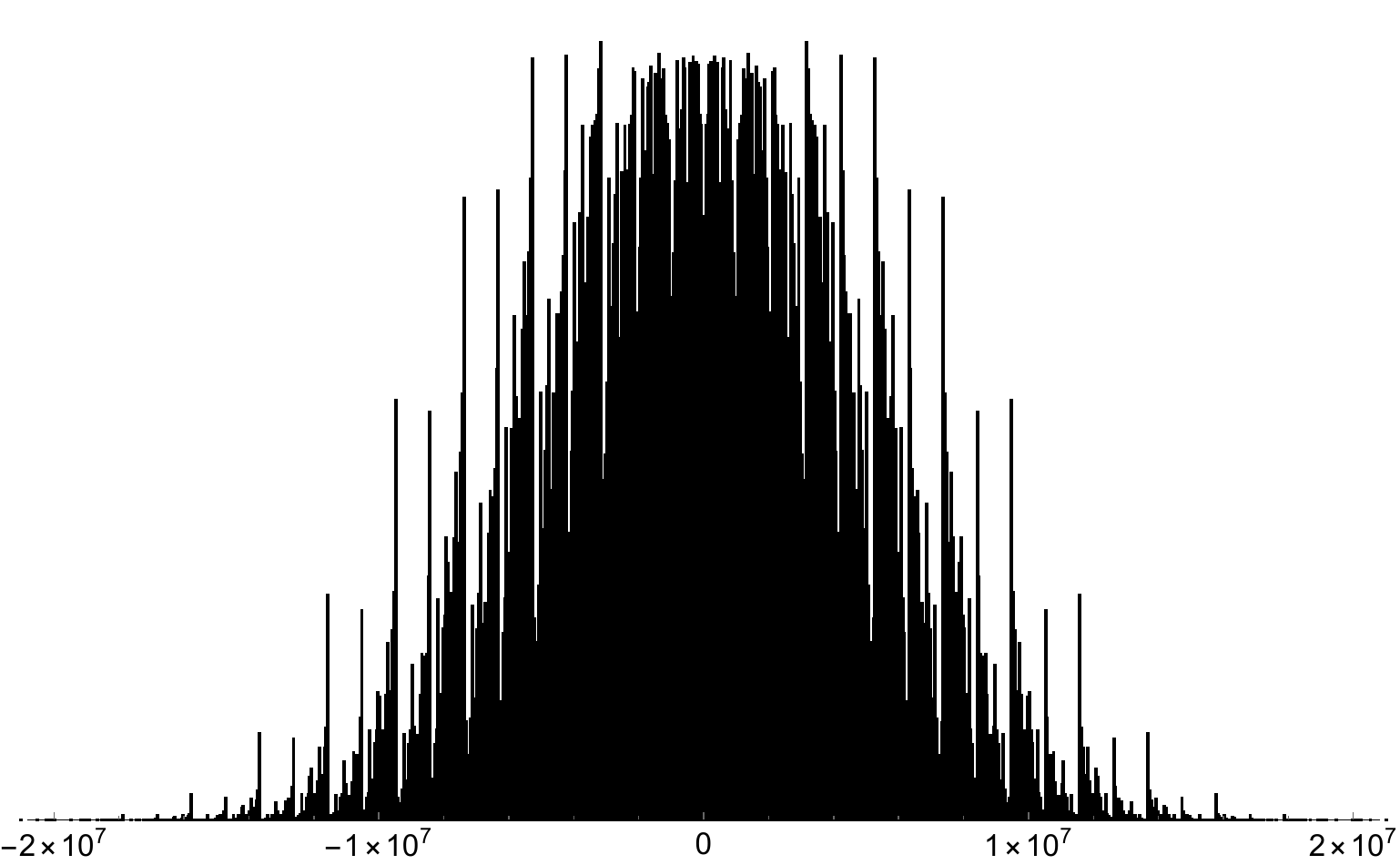}
\caption{A histogram of the discrete measure $\mu$ derived from the first 22 terms from the Conway-Guy sequence.  }
\end{figure}
\end{center}
\vspace{-0pt}
We start by trying to understand which Gaussian we should compare the distribution $\mu$ to. A Gaussian is uniquely determined by mean and variance.
 Since $\mu$ is symmetric around the origin, the expectation is $\mathbb{E}X = 0$.  Simultaneously, we
have an explicit expression for the variance and
$$ \mathbb{E} \left(X^2\right) = \mathbb{E} \left(\sum_{i=1}^{n} \varepsilon_i a_i \right)^2 = \mathbb{E} \sum_{i,j=1}^{n} \varepsilon_i \varepsilon_i  a_i a_j = \sum_{i=1}^{n} a_i^2.$$
The probability density function of that Gaussian will be abbreviated as
$$ \gamma(x) = \frac{1}{\sqrt{2\pi}} \left( \sum_{i=1}^{n} a_i^2 \right)^{-1/2} \exp\left( -\frac{x^2}{2} \left( \sum_{i=1}^{n} a_i^2 \right)^{-1}\right).$$
Note that $\gamma$ is a smooth function while $\mu$ is a singular measure. To facilitate a comparison between the two, we will introduce a smoothed version of $\mu$. 
Consider the normalized characteristic function $ h(x) = (1/2) \cdot\chi_{[-1,1]}.$
Since both $\mu$ and $h$ are probability measures, their convolution 
$$ (h * \mu)(x) = \frac{1}{2^n} \sum_{i=1}^{2^n} \frac{1}{2} \chi_{[x_i - 1, x_i + 1]}(x)$$
is also a probability measure. We observe that $h*\mu$ is a sum of characteristic functions centered at the points $x_i$ at which $\mu$ is supported. Since $\mu$ is distributed over exponentially large scales, smoothing at scale 1 does not change any relevant characteristics. With this language in place, the second main result is as follows.

\begin{theorem}Let $c>0$. Suppose $\left\{a_1, \dots, a_n\right\} \subset \mathbb{R}_{>0}$ has 1-separated subset sums and
$a_n^2 \leq c \cdot n^{-1/2} \sum_{i=1}^{n} a_i^2.$
Then, as $n \rightarrow \infty$, we have
$$\int_{\mathbb{R}} \left( (h*\mu)(x) - \gamma(x) \right)^2 dx =  \int_{|x| \geq \frac{1}{4a_n}} \left( \frac{\sin{2\pi x}}{2\pi  x} \right)^2  \prod_{i=1}^{n} \cos{( 2\pi a_i x)^2} dx +o(2^{-n}).$$
\end{theorem}

We emphasize that $h*\mu$ only assumes the values 0 and $2^{-n-1}$ (and the second value is assumed on $2^n$ intervals of length 2). This implies that
$$ \int_{\mathbb{R}} (h*\mu)(x)^2 dx = \frac{1}{2^{n+1}} \qquad \mbox{while} \qquad \int_{\mathbb{R}} \gamma(x)^2 dx = \frac{1}{2 \sqrt{\pi}} \left(\sum_{i=1}^{n}a_i^2\right)^{-1/2}.$$
We also remark that the probability density of the random walk behaving similarly to a Gaussian was already used by Dubroff, Fox and Xu who invoked the Berry-Esseen theorem. Under a slightly stronger assumption ($a_n^2 \leq c \cdot n^{-2/3-\varepsilon} \sum_{i=1}^{n} a_i^2$) the Berry-Esseen theorem  guarantees that 
$$ \sup_{x \in \mathbb{R}} \left| \mu\left([-\infty,x]\right) - \int_{-\infty}^{x} \gamma(y) dy \right| = o(1)$$
which shows convergence of the \textit{cumulative} distribution functions. Theorem 2 establishes that sets with distinct subset sums satisfy (using Theorem 1)
$$\int_{\mathbb{R}} \left( (h*\mu)(x) - \gamma(x) \right)^2 dx \leq \frac{1+o(1)}{2^n}$$
measuring proximity of the probability density functions in the $L^2-$sense.

\subsection{Concluding Remarks.}
Theorem 2 has a fascinating implication insofar as it allows us to reinterpret the Erd\H{o}s distinct subset sums problem (the general version with real numbers being 1-separated) as a genuine problem in probability theory asking whether particularly excellent random walks exist. More precisely, are there positive real numbers $a_1, \dots, a_n > 0$ such that the random unbiased random walk $X = \pm a_1 \pm a_2 \dots \pm a_n$ has, simultaneously, (1) a large standard deviation, (2) a small largest element $a_n$ and (3) the ability to approximate the normal distribution very well in a concrete sense?\\

\begin{quote}
\textbf{Problem.} Fix $c>0$. As $n \rightarrow \infty$, are there random walks  $X = \pm a_1 \pm a_2 \dots \pm a_n$ such that the largest step size is small compared to the variance
$$ \mbox{largest stepsize} = a_n \leq c \cdot n^{-1/3} \sqrt{ \mathbb{V} X}.$$
and, simultaneously, $X$ has a large variance and approximates a Gaussian well in the sense of
$$ \frac{1}{2 \sqrt{\pi}} \frac{1}{\sqrt{\mathbb{V} X}} + \int_{\mathbb{R}} \left( (h*\mu)(x) - \gamma(x) \right)^2 dx =  \frac{1+o(1) }{2^{n+1}}~?$$
\end{quote}

If there exist $\left\{a_1, \dots, a_n\right\} \subset \mathbb{N}$ with distinct subset sums and $a_n \lesssim n^{-1/3-\varepsilon} \cdot 2^n$, then such random walks do indeed exists: this follows from combining Theorem 1, Theorem 2 and the computation carried out after the proof of Lemma 1.\\

Note that, considering the constraint on $a_n$ being as small as possible and considering the structure of the first term, it does seem like one would like to have many of the $a_i$ to be roughly comparable to $a_n$. The Conway-Guy \cite{conway} sequence has this property: for each $\varepsilon > 0$ at least $n - c_{\varepsilon} \log{n}$ terms satisfy $a_i \geq (1-\varepsilon) a_n$. We also observe that for sets of that type, where many of the $a_i$ are comparable in size to $a_n$, one can draw additional information from Theorem 1 
$$ \int_{\mathbb{R}} \left( \frac{\sin{ x}}{  x} \right)^2  \prod_{i=1}^{n} \cos{( a_i x)^2} dx = \frac{\pi}{2^{n}}.$$
The cosines are all aligned at $x=0$, the contribution to the integral coming from close to the origin is really just a function of $\sum_{i=1}^{n} a_i^2$ (see the comment after the proof of Lemma 1) and fairly independent of the arithmetic structure. The next interesting point is $x = \pi/a_n$: if we have $a_i = (1+o(1)) a_n$ for many $1 \leq i \leq n-1$, then many of the cosines will still be aligned at $\pi/a_n$. The only way to avoid a large contribution is to have an $a_i \sim (1+o(1)) a_n/2$. So it is not inconceivable that Theorem 1 suggests a sort of multi-scale structure as being possibly favorable. The argument can then be continued for $x = k \pi/a_n$ for small $k \in \mathbb{Z}$. As $k$ gets larger, one would expect the cosines to decorrelate.

\section{Proofs}
\subsection{Proof of Theorem 1}
\begin{proof}
As already mentioned, we will smooth $\mu$ by convolving with the normalized characteristic function
$ h(x) = (1/2) \cdot \chi_{[-1,1]}.$
Since both $\mu$ and $h$ are probability measures, their convolution 
$$ (h * \mu)(x) = \frac{1}{2^n} \sum_{i=1}^{2^n} \frac{1}{2} \chi_{[x_i - 1, x_i + 1]}$$
is also a probability measure. We observe that $h*\mu$ is a sum of characteristic functions and its $L^1-$norm is $1$. Its $L^2-$norm is minimized if
and only if these characteristic functions do not overlap which is equivalent to $\min_{i \neq j} |x_i - x_j| \geq 2$ and therefore, in turn, equivalent to all subset sums being at least distance 1 from each other. Formally,
\begin{align*}
 \| h * \mu \|_{L^2}^2 &= \frac{1}{4^n}  \int_{\mathbb{R}} \left(\sum_{i=1}^{2^n} \frac{1}{2} \chi_{[x_i - 1, x_i + 1]}\right)^2 dx \\
 &=\frac{1}{4^n}  \int_{\mathbb{R}} \sum_{i,j=1}^{2^n} \frac{1}{2} \chi_{[x_i - 1, x_i + 1]} \frac{1}{2} \chi_{[x_j - 1, x_j + 1]}  dx \\
 &\geq \frac{1}{4^n}  \int_{\mathbb{R}} \sum_{i=1}^{2^n} \frac{1}{2} \chi_{[x_i - 1, x_i + 1]} \frac{1}{2} \chi_{[x_i - 1, x_i + 1]}  dx \\
 &= \frac{1}{4^n} 2^{n-1} = \frac{1}{2^{n+1}}.
 \end{align*}
 This is the only inequality in the entire argument and is attained if and only if all $x_i$ are $2-$separated.
Using that the Fourier transform is unitary on $L^2$ and sends convolution to products, 
$$ \| \mu * h\|_{L^2}^2 = \| \widehat{\mu * h} \|_{L^2}^2 = \| \widehat{\mu} \widehat{h} \|_{L^2}^2 = \int_{\mathbb{R}} \widehat{h}(\xi)^2 \widehat{\mu}(\xi)^2 d\xi.$$
It remains to compute the Fourier transforms: the Fourier transform of the characteristic function $h$ is completely explicit
$$ \widehat{h}(\xi) = \frac{\sin{(2\pi \xi)}}{2\pi \xi}.$$
The measure $\mu$ can itself be defined as a convolution
$$ \mu = \left( \frac{\delta_{-a_1}}{2} + \frac{\delta_{a_1}}{2}\right)* \left( \frac{\delta_{-a_2}}{2} + \frac{\delta_{a_2}}{2}\right) * \dots *  \left( \frac{\delta_{-a_n}}{2} + \frac{\delta_{a_n}}{2}\right).$$
Using again that the Fourier transform sends convolution to products and
$$ \widehat{ \left( \frac{\delta_{-a_i}}{2} + \frac{\delta_{a_i}}{2}\right)}(\xi) = \frac{e^{2\pi i (-a_i) \xi}}{2} +  \frac{e^{2\pi i a_i \xi}}{2} = \cos{(2\pi a_i \xi)}$$
leads to
$$ \widehat{\mu}(\xi) = \prod_{i=1}^{n} \cos{(2\pi a_i \xi)}.$$
Thus 
$$ \int_{\mathbb{R}} \left( \frac{\sin{2\pi x}}{2\pi  x} \right)^2  \prod_{i=1}^{n} \cos{( 2\pi a_i x)^2} dx \geq \frac{1}{2^{n+1}}$$
with equality if and only if all subset sums of $\left\{a_1, \dots, a_n\right\}$ are 1-separated.\end{proof}

\subsection{Proof of Corollary 1}
\begin{proof}
If all $a_i$ are integers, then the product is $1-$periodic and, together with
$$ \sum_{k \in \mathbb{Z}}  \left( \frac{\sin{2\pi (x-k)}}{2\pi  (x-k)} \right)^2 = \frac{1+\cos{(2\pi x)}}{2},$$
this implies
$$2 \int_{\mathbb{R}} \left( \frac{\sin{2\pi x}}{2\pi  x} \right)^2  \prod_{i=1}^{n} \cos{( 2\pi a_i x)^2} dx = \int_0^1  (1+\cos{(2\pi x)})\prod_{i=1}^{n} \cos{( 2\pi a_i x)^2} dx.$$
To further evaluate the integral, we switch back to exponentials and note that
$$  \cos{( 2\pi a_i x)^2}  = \frac{1}{2} + \frac{e^{4 \pi i a_i x} + e^{-4 \pi i a_i x}}{4}$$
leading to the integral
$$ \frac{1}{2^{n}}  \int_0^1  \left(1 + \frac{e^{2 \pi i x} + e^{-2\pi i x}}{2} \right) \prod_{i=1}^{n} \left(  1 + \frac{e^{4 \pi i a_i x} + e^{-4 \pi i a_i x}}{2} \right).$$
Selecting the constant 1 in all terms leads to a contribution of $2^{-n}$. Any other choice of combinations from the big product leads to exponentials of the form
$\exp( 4 \pi i k x)$ where $k \in \mathbb{Z} \setminus \left\{0 \right\}$ whenever all subset sums are distinct.  Thus every other contributions leads to 0 and
$$ \int_0^1  \prod_{i=1}^{n} \cos{( 2\pi a_i x)^2} dx \geq \frac{1}{2^n}$$
if all subset sums are distinct. Conversely, if not all subset sums are distinct, then there is a corresponding choice of combinations in the product leading to a zero frequency: since all coefficients are nonnegative, we see that the integral will then be larger than $2^{-n}$.
\end{proof}

\subsection{Proof of Lemma 1}
\begin{proof}[Proof, close to Elkies \cite{elk}] Note that, for example, $a_n \geq 2^n/n$, already implies that the interval is very close to the origin where $\sin{(2\pi x)}/(2\pi x) \sim 1$ and thus
$$  \int_{|x| \leq \frac{1}{4a_n}} \left( \frac{\sin{2\pi x}}{2\pi  x} \right)^2  \prod_{i=1}^{n} \cos{( 2\pi a_i x)^2} dx = \left(1 -o(1))\right)  \int_{|x| \leq \frac{1}{4a_n}}   \prod_{i=1}^{n} \cos{( 2\pi a_i x)^2} dx.$$
On this interval, we have, for all $1 \leq i \leq n-1$ that
 $ \cos{(2 \pi a_i x)} \geq  \cos{( 2 \pi a_n x)}$
and
 \begin{align*}
 \int_{|x| \leq \frac{1}{4a_n} }  \prod_{i=1}^{n} \cos{( 2 \pi a_i x)^2} dx \geq  \int_{|x| \leq \frac{1}{4 a_n}} \cos{( 2 \pi a_n x)^{2n}} dx
 \end{align*}
 A change of variables and evaluating the integral (see \cite{elk}) shows that
\begin{align*}
 \int_{|x| \leq \frac{1}{4 a_n} }  \cos{( 2 \pi a_n x)^{2n}} dx &= \frac{1}{2 \pi a_n} \int_{|x| \leq  \frac{\pi}{2}}  \cos{(  x)^{2n}} dx \\
 &= \frac{1}{ 2 \pi a_n}  \frac{\pi}{4^n} \binom{2n}{n} \\
 &= (1+o(1)) \frac{1}{2} \frac{1}{ a_n} \frac{1}{\sqrt{ \pi n}},
 \end{align*}
 where evaluating $\int \cos{(x)}^{2n} dx$ in terms of binomial coefficients is classical \cite{grad}.
\end{proof}
This implies a lower bound on $a_n$ since
\begin{align*}
 \frac{1}{2^{n+1}} &=\int_{\mathbb{R}} \left( \frac{\sin{2\pi x}}{2\pi  x} \right)^2  \prod_{i=1}^{n} \cos{( 2\pi a_i x)^2} dx\\
&\geq  \int_{|x| \leq \frac{1}{4a_n}} \left( \frac{\sin{2\pi x}}{2\pi  x} \right)^2  \prod_{i=1}^{n} \cos{( 2\pi a_i x)^2} dx \geq (1+o(1)) \frac{1}{2} \frac{1}{ a_n} \frac{1}{\sqrt{ \pi n}}
\end{align*}
showing that $a_n \geq 2^n/\sqrt{\pi n}$ which is, in spirit, the original argument of Elkies. We note that, provided $a_n$ is
small, i.e. $a_n = o(\sum_{i=1}^{n}{a_i^2})$, one can Taylor expand the cosines and, for $x$ small,
$$   \prod_{i=1}^{n} \cos{( 2\pi a_i x)^2} dx\sim \exp\left( - 4\pi^2 x^2 \sum_{i=1}^{n} a_i^2\right)$$
which then leads to the slightly refined estimate
$$  \int_{|x| \leq \frac{1}{4 a_n} }   \prod_{i=1}^{n} \cos{( 2\pi a_i x)^2} dx dx  \geq \frac{(1+o(1))}{2 \sqrt{\pi}} \left( \sum_{i=1}^{n} a_i^2 \right)^{-1/2}.$$
At this point, we do not know of any argument that excludes the possibility that $n-o(n)$ of the $a_i$ satisfy $a_i = (1+o(1)) a_n$ and
this refined estimate does not currently lead to any information different from that provided by the cruder estimate above.
Indeed, the Conway-Guy sequence is an example of a set with distinct subset sums and this type of behavior, perhaps extremal configurations do behave like that.

\subsection{Technical Lemma.}
The goal of this section is to establish an upper bound on the difference between $\widehat{\mu}$ and the approximating Gaussian measure close to the origin. Lemma 3 will then quickly imply Theorem 2.
\begin{lemma} 
Let $c>0$. Suppose $\left\{a_1, \dots, a_n\right\} \subset \mathbb{R}_{>0}$ has 1-separated subset sums and
$a_n^2 \leq c \cdot n^{-1/2} \sum_{i=1}^{n} a_i^2.$
Then, as $n \rightarrow \infty$, we have
$$ \int_{|x| \leq \frac{1}{4a_n}} \left| \frac{\sin{(2\pi x)}}{2 \pi x}  \prod_{i=1}^{n} \cos{( 2\pi a_i x)} -   \exp\left( -2 \pi^2 x^2  \sum_{i=1}^{n} a_i^2 \right) \right|^2 dx = o(2^{-n}).$$
\end{lemma}
\begin{proof} 
The first step is a Taylor expansion around $x=0$
\begin{align*}
 \frac{\sin{(2\pi x)}}{2 \pi x} \prod_{i=1}^{n} \cos{( 2\pi a_i x)}  &= \frac{\sin{(2\pi x)}}{2 \pi x} \exp\left(  \sum_{i=1}^{n} \log\left( \cos{( 2\pi a_i x)}  \right) \right) \\
 &\leq\frac{\sin{(2\pi x)}}{2 \pi x}  \exp\left(  \sum_{i=1}^{n} \log\left( 1 - 2 \pi^2 a_i^2 x^2 + \mathcal{O}(a_i^4 x^4)  \right) \right)\\
  &=\frac{\sin{(2\pi x)}}{2 \pi x}  \exp\left(  \sum_{i=1}^{n}  - 2 \pi^2 a_i^2 x^2 + \mathcal{O}(a_i^4 x^4)   \right) \\
  &=  e^{ \mathcal{O}(x^2 + n a_n^4 x^4)   }  \exp\left( -2 \pi^2 x^2  \sum_{i=1}^{n} a_i^2 \right).
\end{align*}
 The goal is to bound 
$$X = \int_{|x| \leq \frac{1}{4a_n}} \left|  \frac{\sin{(2\pi x)}}{2 \pi x}  \prod_{i=1}^{n} \cos{( 2\pi a_i x)} -   \exp\left( -2 \pi^2 x^2  \sum_{i=1}^{n} a_i^2 \right) \right|^2 dx$$
which, considering asymptotic expansion, can be bounded as
\begin{align*}
X \leq \int_{|x| \leq \frac{1}{4a_n}} \left|  e^{  \mathcal{O}(x^2 + n a_n^4 x^4)   }  - 1 \right|^2  \exp\left( -4 \pi^2 x^2  \sum_{i=1}^{n} a_i^2 \right) dx.
\end{align*}
This bound by itself is a little bit too crude but is reasonably close to the origin:
note that the integrand is the product of two functions the first of which is small for small values of $x$ and the second of which is small for large $x$.  This suggests splitting the integral into the regions: for some $0 < \delta < 1/(4a_n)$ to be optimized later, we write
$ I_1 = \left\{x : |x| \leq \delta \right\}$ and let  $I_2 = \left\{x : \delta  \leq  |x| \leq 1/(4 a_n) \right\}.$
Provided $\delta^2 + n a_n^2 \delta^4 = \mathcal{O}(1)$, we can estimate the integral over $I_1$ as
\begin{align*}
 Y &=   \int_{I_1} \left|  e^{  \mathcal{O}(x^2 + n a_n^4 x^4)   }  - 1 \right|^2  \exp\left( -4 \pi^2 x^2  \sum_{i=1}^{n} a_i^2 \right) dx \\
 &\leq \mathcal{O}(\delta^2 + n a_n^4 \delta^4) \cdot \int_{\mathbb{R}} \exp\left( -4 \pi^2 x^2  \sum_{i=1}^{n} a_i^2 \right) dx\\
  &=  \mathcal{O}(\delta^2 + n a_n^4 \delta^4)  \cdot \frac{1}{2 \sqrt{\pi}}\left(\sum_{i=1}^{n} a_i^2 \right)^{-1/2}.
  \end{align*}
We use a different type of expansion for the second region: note that, for $|x| < \pi/2$,
$$ \log\left(\cos{(x)}\right) \leq - \frac{x^2}{2}$$
and thus, for $|x| \leq1/(4a_n)$,
$ \cos{( 2\pi a_i x)}  \leq \exp\left( -2 \pi^2 x^2 a_i^2 \right)$
from which we deduce
$$ \forall |x| \leq \frac{1}{4a_n} \qquad  0 \leq \prod_{i=1}^{n} \cos{( 2\pi a_i x)} \leq \exp\left( -2 \pi^2 x^2  \sum_{i=1}^{n} a_i^2 \right).$$
Therefore the contribution of the integrand to $X$ over $I_2$, which is
$$  Z = \int_{I_2}  \left|  \exp\left( -2 \pi^2 x^2  \sum_{i=1}^{n} a_i^2 \right) -  \frac{\sin{(2\pi x)}}{2 \pi x}  \prod_{i=1}^{n} \cos{( 2\pi a_i x)}  \right|^2 dx,$$
 can be trivially bounded from above by
\begin{align*} 
 Z &\leq \int_{I_2}  \exp\left( -2 \pi^2 x^2  \sum_{i=1}^{n} a_i^2 \right)^2 dx \leq \frac{2}{\delta}  \int_{\delta}^{\infty} x \exp\left( -4 \pi^2 x^2  \sum_{i=1}^{n} a_i^2 \right) dx \\
 &= \frac{2}{\delta} \frac{1}{8 \pi^2 \sum_{i=1}^{n} a_i^2}  \exp\left( - 4 \delta^2 \pi^2 \sum_{i=1}^{n}{a_i^2}\right) \leq \frac{1}{\delta} \frac{1}{\sum_{i=1}^{n} a_i^2} \exp\left(- \delta^2 \sum_{i=1}^{n}a_i^2 \right).
\end{align*}
We want all error estimates to be $o(2^{-n})$ and achieve this by setting
$$ \delta = \alpha_n  \left(\sum_{i=1}^{n} a_i^2 \right)^{-1/2} $$
with $\alpha_n$ an arbitrarily slowly growing sequence (think of $\alpha_n = \log \log \log n$). We start by checking whether our first asymptotic expansion is valid in this regime, i.e. whether $\delta^2 + n a_n^2 \delta^4 = \mathcal{O}(1)$. Moser's estimate implies $\sum_{i=1}^{n} a_i^2 \gtrsim 4^n$ and thus
$$ \delta^2 + n a_n^2 \delta^4 = \mathcal{O}(\alpha_n^2 4^{-n}) + \mathcal{O}( n a_n^2 \alpha_n^4 4^{-2n})  = o(2^{-n}).$$
The next step is an estimate on $Y$. Importing our upper bound on $a_n$ shows
\begin{align*}
 Y &\leq  \mathcal{O}(\delta^2 + n a_n^4 \delta^4)  \left(\sum_{i=1}^{n} a_i^2 \right)^{-1/2} \lesssim n a_n^4 \alpha_n^4 \left(\sum_{i=1}^{n} a_i^2 \right)^{-5/2}\lesssim n^{-\varepsilon}  \left(\sum_{i=1}^{n} a_i^2 \right)^{-1/2} 
\end{align*}
which is $\mathcal{O}(n^{-\varepsilon} 2^{-n}) = o(2^{-n})$. Finally, for the last error term, 
$$ Z \leq \frac{1}{\delta} \frac{1}{\sum_{i=1}^{n} a_i^2} \exp\left(- \delta^2 \sum_{i=1}^{n}a_i^2 \right) \leq 10 \frac{2^{-n}}{\alpha_n} e^{-\alpha_n^2} = o(2^{-n}).$$
This proves Lemma 3.
\end{proof}

\subsection{Proof of Theorem 2}
\begin{proof} Theorem 2 is a relatively easy consequence of Lemma 3.
Recall that
$$ \gamma(x) = \frac{1}{\sqrt{2\pi}} \left( \sum_{i=1}^{n} a_i^2 \right)^{-1/2} \exp\left( -\frac{x^2}{2} \left( \sum_{i=1}^{n} a_i^2 \right)^{-1}\right)$$
and $\widehat{\gamma}(x) = \exp(-2\pi^2 x^2 \sum_{i=1}^{n} a_i^2 )$.
Using the Fourier transform we get that
$$ X = \int_{\mathbb{R}} | (h*\mu)(x) - \gamma(x)|^2 dx = \int_{\mathbb{R}} | (\widehat{h*\mu})(x) - \widehat{\gamma}(x)|^2 dx$$
can be written as
$$ X =  \int_{\mathbb{R}} \left| \frac{\sin{(2\pi x)}}{2 \pi x}  \prod_{i=1}^{n} \cos{( 2\pi a_i x)} -   \exp\left( -2 \pi^2 x^2  \sum_{i=1}^{n} a_i^2 \right) \right|^2 dx.$$
Lemma 3 implies that if $\left\{a_1, \dots, a_n\right\} \subset \mathbb{R}_{>0}$ has 1-separated subset sums and
$a_n^2 \leq c \cdot n^{-1/2} \sum_{i=1}^{n} a_i^2,$
then
$$ \int_{|x| \leq \frac{1}{4a_n}} \left| \frac{\sin{(2\pi x)}}{2 \pi x}  \prod_{i=1}^{n} \cos{( 2\pi a_i x)} -   \exp\left( -2 \pi^2 x^2  \sum_{i=1}^{n} a_i^2 \right) \right|^2 dx = o(2^{-n})$$
implying that, by splitting the integral into $\left\{|x| \leq 1/(4 a_n) \right\}$ and  $\left\{|x| \geq 1/(4 a_n) \right\}$,
$$ X =  \int_{|x| \geq \frac{1}{4a_n}} \left| \frac{\sin{(2\pi x)}}{2 \pi x}  \prod_{i=1}^{n} \cos{( 2\pi a_i x)} -   \exp\left( -2 \pi^2 x^2  \sum_{i=1}^{n} a_i^2 \right) \right|^2 dx +o (2^{-n}).$$
Using the upper bound on $a_n$
\begin{align*}
\int_{|x| \geq 1/(4 a_n)}  \exp\left( -4 \pi^2 x^2  \sum_{i=1}^{n} a_i^2 \right) dx &\leq 8 a_n \int_{1/(4 a_n)}^{\infty} x  \exp\left( -4 \pi^2 x^2  \sum_{i=1}^{n} a_i^2 \right) dx\\
&\leq \frac{a_n}{\sum_{i=1}^{n} a_i^2} \exp\left( - \frac{\pi^2}{4} \frac{\sum_{i=1}^{n}a_i^2}{a_n^2} \right) \\
&\lesssim \frac{1}{n^{1/4}} \frac{1}{ \left(\sum_{i=1}^{n} a_i^2\right)^{1/2}} e^{-c \sqrt{n}}.
\end{align*} 
 
With Moser's estimate 
$ \sum_{i=1}^{n} a_i^2 \gtrsim 4^n$
one deduces $$ \left\|   \exp\left( -2 \pi^2 x^2  \sum_{i=1}^{n} a_i^2 \right)  \chi_{|x| \geq \frac{1}{4a_n}} \right\|_{L^2(\mathbb{R})} \lesssim e^{-c \sqrt{n}} \frac{1}{2^n}.$$
Using the triangle inequality in $L^2$, we see that
 \begin{align*}
 Z &=  \left\|  \left[ \frac{\sin{(2 \pi x)}}{2\pi x} \prod_{i=1}^{n} \cos{( 2\pi a_i x)} -  \exp\left( -2 \pi^2 x^2  \sum_{i=1}^{n} a_i^2 \right) \right]  \chi_{|x| \geq \frac{1}{4a_n}} \right\|_{L^2(\mathbb{R})} \\
 &= \left\|   \frac{\sin{(2 \pi x)}}{2\pi x} \prod_{i=1}^{n} \cos{( 2\pi a_i x)}  \chi_{|x| \geq \frac{1}{4a_n}} \right\|_{L^2(\mathbb{R})} + \mathcal{O}\left(e^{-c \sqrt{n}} \frac{1}{2^n} \right).
 \end{align*}
Squaring both sides and using Theorem 1, we deduce
$$ X = Z^2 + o(2^{-n}) = \int_{|x| \geq \frac{1}{4a_n}} \left( \frac{\sin{2\pi x}}{2\pi  x} \right)^2  \prod_{i=1}^{n} \cos{( 2\pi a_i x)^2} dx + o(2^{-n}).$$
\end{proof}

\subsection{Proof of Lemma 2}  
Throughout this proof, we will abbreviate
$$ \gamma(x) = \frac{1}{\sqrt{2\pi}} \left( \sum_{i=1}^{n} a_i^2 \right)^{-1/2} \exp\left( -\frac{x^2}{2} \left( \sum_{i=1}^{n} a_i^2 \right)^{-1}\right)$$
for the Gaussian approximating $\mu$.
Before proving Lemma 2, we quickly recall the Berry-Esseen theorem which, in our setting, says that
$$ \sup_{x \in \mathbb{R}} \left| \mu([-\infty,x]) - \int_{-\infty}^{x} \gamma(x) dx \right| \leq  \frac{\sum_{i=1}^{n} a_i^3}{\left(\sum_{i=1}^{n} a_i^2 \right)^{3/2} } .$$
Assuming that $a_n^2 = \mathcal{O}\left( n^{-2/3 -\varepsilon} \sum_{i=1}^{n} a_i^2\right)$, one can bound this by
$$  \frac{\sum_{i=1}^{n} a_i^3}{\left(\sum_{i=1}^{n} a_i^2 \right)^{3/2} } \lesssim \frac{n \cdot a_n^3}{ {\left(\sum_{i=1}^{n} a_i^2 \right)^{3/2} }} \lesssim n^{-\frac{3 \varepsilon}{2}} = o(1).$$
The way we will use this information is that, for any interval $J \subset \mathbb{R}$
$$ \left| \mu(J) - \int_{J} \gamma(x) dx \right| = o(1).$$
Note that this argument was also used by Dubroff, Fox and Xu \cite{dubroff} for $J$ an interval centered at the origin whose length is proportional to a small multiple of the standard deviation of the Gaussian. We will quickly summarize their short argument at an appropriate place in the proof of Lemma 2.
\begin{proof}[Proof of Lemma 2] 
We start the argument with a lower bound on
$$ X = \int_{\mathbb{R}} \left| (\mu * h)(x)  - \gamma(x) \right|^2 dx.$$
Taking a Fourier transform,
 $$ X = \int_{\mathbb{R}}  \left|  \frac{\sin{(2 \pi x)}}{2\pi x} \prod_{i=1}^{n} \cos{( 2\pi a_i x)} -   \exp\left( -2 \pi^2 x^2  \sum_{i=1}^{n} a_i^2 \right) \right|^2 dx.$$
We split the integral into two regions: $|x| \leq 1/(4a_n)$ and the remaining region. Lemma 3 implies that the integral over the first region is $o(2^{-n})$, it remains to analyze the integral over the second region. Arguing exactly as in the proof of Theorem 2, we deduce that
$$ X = \int_{|x| \geq \frac{1}{4a_n}} \left( \frac{\sin{2\pi x}}{2\pi  x} \right)^2  \prod_{i=1}^{n} \cos{( 2\pi a_i x)^2} dx + o(2^{-n}).$$
The next argument is completely independent of all the previous arguments: we will derive a lower bound on the same quantity
via a completely different argument which will then imply Lemma 2. Recall that
$$ X = \int_{\mathbb{R}} \left| (\mu * h)(x) - \gamma(x)  \right|^2 dx.$$
 $\mu *h$ only assumes the values $\left\{0, 2^{-n-1} \right\}$.
 Moreover, by the argument above, 
\begin{align*}
\sup_{J \subset \mathbb{R} \atop J~\mbox{\tiny interval}} \left| \int_{J} (\mu * h)(x) dx - \int_{J}\gamma(x) dx \right| = o(1).
 \end{align*}
 This leads to an amusing setting: we know that $\mu*h$ approximates the Gaussian in probability over intervals. Simultaneously, $\mu*h$ can only assume two values one of which is 0: thus, the local density of the Gaussian predicts the density of intervals in the region where $\mu*h$ assumes its nonzero value $2^{-n-1}$. An example of what this could look like is shown in Fig. 2. We conclude with a simple proposition.
 
 \begin{proposition} Let $\mu$ be the probability density function of a $\mathcal{N}(0,\sigma^2)$ Gaussian. Let $(\nu_n)_{n}$ be a sequence of probability density functions such that
 \begin{enumerate}
 \item $\nu_n \rightarrow \mu$ in probability: for every interval $J \subset \mathbb{R}$ we have
 $$ \lim_{n \rightarrow \infty} \int_J \nu_n(x) dx = \int_J \mu(x) dx$$
 \item and $\nu_n(x)$ only assumes two values $\left\{0, z_n\right\}$ for some $z_n > 0$. 
 \end{enumerate}
 Then
 $$ \liminf_{n \rightarrow \infty} \int_{\mathbb{R}} (\mu(x) - \nu_n(x))^2 dx \geq \frac{\sqrt{2}-1}{2\sqrt{\pi} \sigma}.$$
 \end{proposition}
 \begin{proof}[Proof of the Proposition] The density of $\mu$ is simply given by
 $$ \mu(x) = \frac{1}{\sqrt{2\pi} \sigma} \exp\left( - \frac{1}{2} \frac{x^2}{\sigma^2} \right).$$
 We note that both Properties combined require (by taking $J$ to be a small interval centered around the origin) that
 $$ \liminf_{n \rightarrow \infty} z_n \geq \max_{x \in \mathbb{R}} \mu(x) = \frac{1}{\sqrt{2\pi} \sigma}.$$
 Let now $J$ be a small interval centered around $x_0 \in \mathbb{R}$, say $J = (x_0 - \varepsilon, x_0 + \varepsilon)$. The two properties combined tell us
what can be expected of $\nu_n$: since
$$ \int_{J} \mu(x) dx = 2 \varepsilon \mu(x_0) + \mathcal{O}(\varepsilon^2)$$
we have
$$  \lim_{n \rightarrow \infty} \int_J \nu_n(x) dx  = \int_{J} \mu(x) dx = 2 \varepsilon \mu(x_0) + \mathcal{O}(\varepsilon^2).$$
This allows us to deduce that the fraction $\alpha$ of the interval $J$ where $\nu_n$ assumes the value $z_n$ and the
remaining fraction $(1-\alpha)$ where it assumes the value 0 is determined by
$$ \alpha z_n = \mu(x_0) + \mbox{lower order terms}.$$
This tells us that
 $$ \int_{\mathbb{R}} (\mu(x) - \nu_n(x))^2 dx = (1+o(1)) \int_{\mathbb{R}} \frac{\mu(x)}{z_n} (\mu(x) - z_n)^2 + \left(1 -  \frac{\mu(x)}{z_n} \right) \mu(x)^2 dx.$$ 
 The integral algebraically simplifies to
  $$\int_{\mathbb{R}} \frac{\mu(x)}{z_n} (\mu(x) - z_n)^2 + \left(1 -  \frac{\mu(x)}{z_n} \right) \mu(x)^2 dx =   \int_{\mathbb{R}} \mu(x) (z_n - \mu(x)) dx.$$
  At this point, we recall that, up to lower order terms, $z_n \geq \mu(0)$. Thus
  $$  \int_{\mathbb{R}} \mu(x) (z_n - \mu(x)) dx \geq  \int_{\mathbb{R}} \mu(x) \left( \frac{1}{\sqrt{2\pi} \sigma} - \mu(x)\right) dx = \frac{1}{\sqrt{2\pi} \sigma} -  \frac{1}{2\sqrt{\pi} \sigma}$$
 which is the desired result.
 \end{proof}

    \begin{center}
\begin{figure}[h!]
\includegraphics[width=0.5\textwidth]{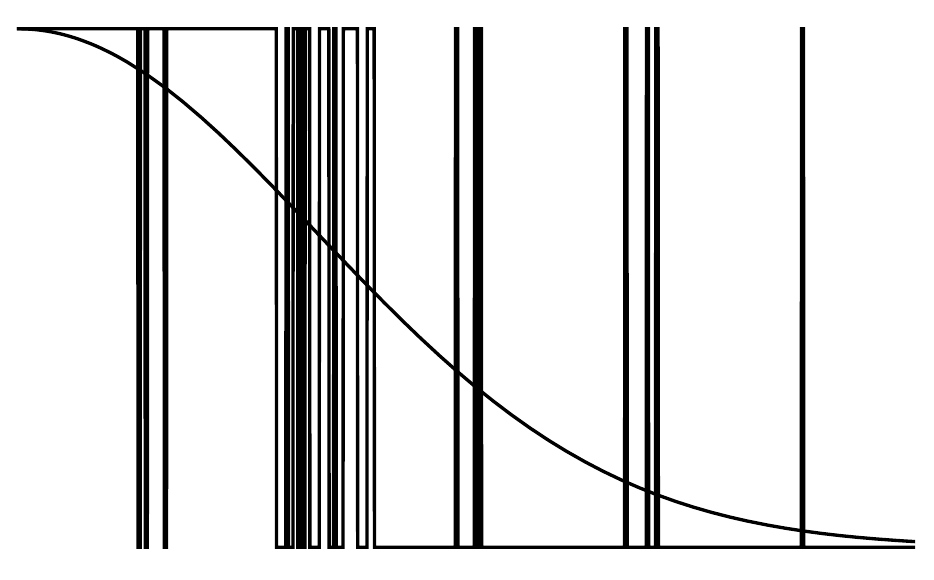}
\caption{A step function assuming only two values approximating a Gaussian density.}
\end{figure}
\end{center}
At this point can we quickly note, in passing, the original argument of Dubroff, Fox and Xu \cite{dubroff}:
the Gaussian attains its maximum density at the origin and therefore
$$\gamma(0) = \frac{1}{\sqrt{2 \pi} }\left( \sum_{i=1}^{n} a_i^2 \right)^{-1/2} \leq (1+o(1)) \cdot \| \mu*h\|_{L^{\infty}} = \frac{1+o(1)}{2^{n+1}}$$
from which one deduces
$$ \sqrt{n} \cdot a_n \geq \left( \sum_{i=1}^{n} a_i^2 \right)^{1/2}  \geq (1+o(1)) \sqrt{\frac{2}{\pi}} \cdot 2^n.$$
We can now conclude by applying the Proposition. The variance $\sigma$ of the mollified random walk is, up to lower order terms, given
by the variance of the random walk which is $\sum_{i=1}^{n} a_i^2$. Thus, applying the Proposition, as $n$ becomes large,
\begin{align*}
 X &\geq  (1+o(1))\frac{\sqrt{2}-1}{2  \sqrt{\pi} \left( \sum_{i=1}^{n} a_i^2 \right)^{1/2}}.
 \end{align*}
\end{proof}


\begin{thebibliography}{10}

\bibitem{aliev}  I. Aliev, Siegel’s lemma and sum-distinct sets, Discrete Comput. Geom. 39 (2008), 59--66.

\bibitem{alon} N. Alon and J. H. Spencer, The probabilistic method, 4th ed. Wiley, Hoboken, NJ, 2016.

\bibitem{bae} J. Bae, On subset-sum-distinct sequences. Analytic number theory, Vol. 1, Progr. Math., 138, Birkh\"auser, Boston,
1996, 31--37.

\bibitem{behr} E. Behrends,  Tupel aus $n$ nat\"urlichen Zahlen, für die alle Summen verschieden sind, und ein Masskonzentrations-Ph\"anomen. Elemente der Mathematik, 74 (2019), p. 114-130.

\bibitem{boh} T. Bohman, A sum packing problem of Erd\H{o}s and the Conway-Guy sequence. Proceedings of the American Mathematical Society, 124 (1968), 3627-3636.

\bibitem{boh2} T. Bohman,  A construction for sets of integers with distinct subset sum, the electronic journal of combinatorics 5 (1998), R3.

\bibitem{bor} P. Borwein and M. Mossinghoff, 
Newman polynomials with prescribed vanishing and integer sets with distinct subset sums. 
Math. Comp. 72 (2003), no. 242, 787--800.

\bibitem{conway}  J.H. Conway and R.K. Guy, Sets of natural numbers with distinct sums, Notices Amer.
Math. Soc. 15 (1968), 345.

\bibitem{dubroff} Q. Dubroff, J. Fox and M. W. Xu, A note on the Erdos distinct subset sums problem. SIAM Journal on Discrete Mathematics, 35 (2021), 322--324.

\bibitem{elk} N. Elkies, An improved lower bound on the greatest element of a sum-distinct set of fixed order, Journal of Combinatorial Theory, Series A
Volume 41 (1986), p. 89--94

\bibitem{erd} P. Erd\H{o}s, Problems and results in additive number theory, Colloque sur la Theorie des Nombres, Bruxelles, 1955, pp.
127–137, George Thone, Liege; Masson and Cie, Paris, 1956.

\bibitem{erd2} P P. Erd\H{o}s, Quelques probl\`emes de théorie des nombres, Monographies de l'Enseignement Mathématique, No. 6 , pp. 81--135, L'Enseignement Mathématique, Université, Geneva, 1963

\bibitem{erd3}
P. Erd\H{o}s, Problems and results on extremal problems in number theory, geometry, and combinatorics, Proceedings of the 7th Fischland Colloquium, I (Wustrow, 1988), Rostock. Math. Kolloq. , No. 38 (1989), 6--14

\bibitem{guy0} R. K. Guy, Unsolved Problems in Intuitive Mathematics, Vol. I, Number Theory, Problem C8, Springer-Verlag (1981).

\bibitem{guy} R. K. Guy, Sets of integers whose subsets have distinct sums, Theory and practice of combinatorics, 141--154, North-Holland Math. Stud., 60, Ann. Discrete Math., 12, North-Holland, Amsterdam, (1982).

\bibitem{grad} I. S. Gradstheyn and I. M. Ryhzik, Table of Integrals, Series, and Products, Academic
Press, New York, 1980.

\bibitem{harper} L. H. Harper, Optimal numberings and isoperimetric problems on graphs, J. Combin. Theory 1 (1966), 385--393.

\bibitem{lind} B. Lindstr\"om, Om ett problem av Erd\H{o}s f\"or talf\"oljder, Nordisk Matematisk Tidskrift
Vol. 16, No. 1/2 (1968), pp. 29-33

\bibitem{lunnon} W. F.  Lunnon, 
Integer sets with distinct subset-sums.
Math. Comp. 50 (1988), 297--320.


\bibitem{malt}  R. Maltby,  Bigger and better subset-sum-distinct sets. 
Mathematika 44 (1997), no. 1, 56--60.

\end{thebibliography}
\end{document}